\documentclass[11pt,reqno]{amsart}
\usepackage[T1]{fontenc}
\usepackage{lmodern}

\usepackage{amsmath,amssymb,amsthm,mathtools}
\usepackage{enumitem}
\usepackage[expansion=false]{microtype}
\usepackage{tikz}
\usepackage{cite}
\usepackage{float}
\usepackage{needspace}
\usepackage[colorlinks=true,linkcolor=blue,citecolor=blue,urlcolor=blue]{hyperref}
\hypersetup{
 pdftitle={Generalized Cesaro Operators Between Hardy Spaces},
 pdfauthor={Pengcheng Tang}
}

\newtheorem{theorem}{Theorem}[section]
\newtheorem{proposition}[theorem]{Proposition}
\newtheorem{lemma}[theorem]{Lemma}
\newtheorem{corollary}[theorem]{Corollary}

\newtheorem{example}[theorem]{Example}

\newcommand{\D}{\mathbb D} 
\newcommand{\T}{\mathbb T}

\newcommand{\C}{\mathbb C}

\newcommand{\norm}[1]{\left\lVert #1\right\rVert}

\newcommand{\Hol}{\operatorname{Hol}}

\title[Generalized Ces\`aro operators between Hardy spaces]
{Complete mapping criteria for generalized Ces\`{a}ro operators between Hardy spaces}

\author{Pengcheng Tang*}
\address{School of Mathematics and Statistics, Hunan University of Science and Technology, Xiangtan, Hunan 411201, China}
\email{www.tang-tpc.com@foxmail.com}

\author{Huayou Xie}
\address{School of Financial Mathematics and Statistics, Guangdong University of
Finance, Guangzhou, Guangdong 510521,  China}
\email{xiehy@gduf.edu.cn }

\keywords{Ces\`aro operator, Hardy space, Carleson measure, Coefficient multiplier
\\
\ \ \ \ \ \  \ \ \ \ $^*$Corresponding Author.
}

\begin{document}

\begin{abstract}
Let $\mu$ be a finite positive Borel measure on $[0,1)$ and let
$\gamma>0$. We establish sharp mapping criteria for the generalized
Ces\`aro operator
\[
  \mathcal C_{\mu,\gamma}f(z)
 =\sum_{n=0}^\infty \mu_n
 \left(\sum_{k=0}^n
 \frac{\Gamma(n-k+\gamma)}{\Gamma(\gamma)(n-k)!}a_k\right)z^n,\qquad z\in\D,
\]
between Hardy spaces, including the source and target $H^\infty$
endpoints. For $0<p<q<\infty$, for
$0<p=q<1$, and for $0<p\le1$ with $q=\infty$, the boundedness  of $ \mathcal C_{\mu,\gamma}: H^p \to H^q$ is
equivalent to $\mu$ being a $(\gamma+1/p-1/q)$-Carleson measure,
without further restrictions on $\gamma$. For $1\le q<p\le\infty$,
set $1/r=1/q-1/p$. In this range, boundedness
and compactness are equivalent to
\[
 \int_0^1
 \left(\frac{\mu([t,1))}{(1-t)^{\gamma-1/r}}\right)^r
 \frac{dt}{1-t}<\infty.
\]
This condition is also equivalent to $F_{\mu,\gamma}\in H^r$ and to
$\sum_{n\ge0}(n+1)^{r\gamma-2}\mu_n^r<\infty$, where
$F_{\mu,\alpha}=\mathcal C_{\mu,\alpha}(1)$ and
$\mu_n=\int_{[0,1)}t^n\,d\mu(t)$.
For $1<p<\infty$, boundedness and compactness from $H^p$ to
$H^\infty$ are characterized by the shifted condition
$F_{\mu,\gamma+1}\in H^{p'}$, where $p'=p/(p-1)$. We therefore obtain a complete boundedness classification of the
generalized Ces\`aro operators $\mathcal C_{\mu,\gamma}$ between
Hardy spaces $H^p$ and $H^q$ for the full range
$0<p,q\le\infty$.
\end{abstract}

\maketitle

\section{Introduction and main results}

Let $\D=\{z\in\C:|z|<1\}$ and let $\Hol(\D)$ denote the space of analytic functions in $\D$. We write $\T=\partial\D$, with normalized arc measure $d\theta/(2\pi)$.

Throughout the paper, $1/\infty=0$ and $\frac{1}{p'}+\frac{1}{p}=1$ when
$1\leq p\leq\infty$.
The notation $X\lesssim Y$ means that $X\le CY$ for a positive constant
independent of the functions, measures, and indices under consideration;
its dependence on the fixed parameters will be indicated when needed.
We write $X\asymp Y$ when both $X\lesssim Y$ and $Y\lesssim X$ hold.

For $0<p<\infty$, the Hardy space $H^p$ consists of those $f\in\Hol(\D)$ such that
$$
||f||_{p}:=\sup_{0\leq \rho<1} M_p(\rho, f)<\infty,
$$
where
$$
M_p(\rho, f)= \left(\frac{1}{2\pi}\int_0^{2\pi}|f(\rho e^{i\theta})|^p d\theta \right)^{1/p}, \ 0<p<\infty,
$$
$$M_{\infty}(\rho, f)=\sup_{|z|=\rho}|f(z)|.$$
For background on Hardy spaces, we refer to \cite{Duren1970}.

The classical Ces\`{a}ro operator $\mathcal {C}$  is defined in $\Hol(\D)$ as follows: If $f(z)=\sum_{n=0}^\infty a_nz^n\in \Hol(\D)$, then
 $$
\mathcal {C}(f)(z)=\sum_{n=0}^\infty\left(\frac{1}{n+1}\sum_{k=0}^n a_k\right)z^n, \  z\in \D.
$$

Its boundedness on $H^p$ for $0<p<\infty$, has been established by several
methods. Siskakis \cite{Siskakis1987,Siskakis1990} developed a
semigroup approach and a direct argument at $p=1$, while Miao
\cite{Miao1992} treated $0<p<1$.  Nowak \cite{no} provided another proof valid for all $0<p<\infty$. These results motivate the study of Ces\`aro-type operators obtained
by replacing the classical averaging coefficients with a general
sequence. A natural question is how the decay and regularity of this
sequence govern the boundedness of the resulting operators between
spaces of analytic functions.

Let $\mu$ be a finite positive Borel measure on $[0,1)$, with moments
\[
 \mu_n=\int_{[0,1)}t^n\,d\mu(t),\qquad n\ge0.
\]

Galanopoulos, Girela and Merch\'an \cite{GGM2022} introduced the
measure-induced operator $\mathcal C_\mu$ by replacing $1/(n+1)$ in the
classical coefficient formula by $\mu_n$. The action of $\mathcal C_\mu$ between distinct spaces of analytic functions has been extensively studied in recent years. See, for instance, \cite{03,ces6,bel,baoo1,blas2}. Many known characterizations of the boundedness  and compactness  of  $\mathcal C_\mu$ involve Carleson-type measures.

In this paper, we consider the more general family introduced by Bao, Sun and Wulan
\cite{BaoSunWulan2022}. For $\gamma>0$, put
\[
 A_n^\eta=\frac{\Gamma(n+\eta)}{\Gamma(\eta)\Gamma(n+1)},
 \qquad n\ge0,\quad\eta>0,
\]
and define
\begin{equation*}\label{eq:def-Cmug}
 \mathcal C_{\mu,\gamma}f(z)
 =\sum_{n=0}^{\infty}\mu_n
   \left(\sum_{k=0}^{n}A_{n-k}^\gamma a_k\right)z^n =\int_{[0,1)}\frac{f(tz)}{(1-tz)^\gamma}\,d\mu(t),
 \qquad z\in\mathbb D.
\end{equation*}
 We write
$\mathcal C_{\mu,1}=\mathcal C_\mu$. Thus $d\mu(t)=dt$ and $\gamma=1$
recover $\mathcal C$, whereas the choice
$d\mu(t)=\gamma(1-t)^{\gamma-1}\,dt$ gives the Ces\`aro averaging operator
\[
 C^\gamma f(z)
 =\gamma\int_0^1
   \frac{f(tz)(1-t)^{\gamma-1}}{(1-tz)^\gamma}\,dt.
\]
The averaging family was studied by Stempak \cite{Stempak1994} and
Andersen \cite{Andersen1996}. See also \cite{ces13,ste,gat,xiao}. 
In particular, Andersen \cite{Andersen1996} proved  that $C^\gamma$ is bounded
on  $H^p$, $0<p<\infty$, for every $\gamma>0$.
This result will be used in the coefficient factorizations below.

For $s>0$, the measure $\mu$ is called an $s$-Carleson measure if
\begin{equation*}\label{eq:s-carleson}
 \mu([t,1))\lesssim(1-t)^s,\qquad 0\le t<1.
\end{equation*}
The case $s=1$ is the ordinary Carleson condition.

Bao, Sun and Wulan
\cite{BaoSunWulan2022} studied the range of
$\mathcal C_{\mu,\gamma}$ on $H^\infty$ in spaces lying between
appropriate mean Lipschitz spaces and the Bloch space.
Related mapping problems have been considered between the Bloch and
Bergman spaces \cite{GuoTangZhang2024}, between weighted Bergman spaces
\cite{GalanopoulosSiskakisZhao2025}, and with analytic Besov target spaces
\cite{Tang2025}. Blasco and Mas \cite{BlascoMas2026} developed a
systematic treatment on mixed norm spaces.

In \cite{GGM2022}, the authors proved
that, for $1\le p<\infty$, $\mathcal C_\mu$ is bounded on $H^p$
if and only if $\mu$ is a Carleson measure. They also characterized
boundedness on $H^\infty$ by
$\int_{[0,1)}(1-t)^{-1}\,d\mu(t)<\infty$.
Blasco \cite{Blasco2024} studied Ces\`aro-type operators induced by
complex Borel measures and obtained further results between distinct
Hardy spaces. For positive measures, these include the characterization
of $\mathcal C_\mu:H^1\to H^q$, $1\le q\le\infty$, by the
$(2-1/q)$-Carleson condition
\cite[Corollaries~4.6 and~4.7]{Blasco2024}.
At the other source endpoint, for $1\le q\le2$, he proved
\[
 \mathcal C_\mu:H^\infty\to H^q\ \text{is bounded}
 \quad\Longleftrightarrow\quad
 \sum_{n=0}^{\infty}(n+1)^{q-2}\mu_n^q<\infty;
\]
see \cite[Corollary~4.11]{Blasco2024}.
The same work gives a sufficient moment condition in the range
$1\le q<p<\infty$ \cite[Theorem~3.15(ii)]{Blasco2024}.
These results provide the relevant comparisons for the endpoint and
lower-triangle characterizations established here.

Motivated by these results, we seek sharp mapping criteria for
$\mathcal C_{\mu,\gamma}$ between distinct Hardy spaces, with
$\gamma>0$ arbitrary. The contrast between the diagonal Carleson
criterion and the endpoint moment conditions raises a basic question:
is boundedness determined by the decay of $\mu([t,1))$ as $t\to 1^-$, or does it require additional summability? In particular, the sufficient condition available for
$0< q<p<\infty$ calls for a necessary and sufficient replacement
that identifies the precise dependence on the Hardy exponents.
The endpoint results lead to two further questions: whether the
mapping problem from $H^\infty$ can be treated uniformly for all
finite  exponents $q\ge1$, and whether the criteria for finite
targets extend to $H^p\to H^\infty$ when $1<p<\infty$.
We investigate these questions within the generalized family,
examining how the kernel order interacts with the source and target
exponents. Alongside boundedness, we determine whether compactness
requires a separate vanishing condition or follows from the
summability inherent in the boundedness criterion.

Our main results distinguish three mapping regimes.
For $0<p<q<\infty$,  $0<p=q<1$ and also for $0<p\le1$ with
$q=\infty$, boundedness is characterized by the
$(\gamma+1/p-1/q)$-Carleson condition for every $\gamma>0$.
Our first theorem addresses this upper-triangular case.

\begin{theorem}\label{thm:upper-improved}
Let $\gamma>0$ and let $\mu$ be a finite positive Borel measure on
$[0,1)$. Assume one of the following:
\begin{enumerate}[label=\textup{(\alph*)},leftmargin=*]
\item $0<p<q<\infty$;
\item $0<p=q<1$;
\item $0<p\le1$ and $q=\infty$.
\end{enumerate}
Put $ s=\gamma+\frac1p-\frac1q.$ Then
$\mathcal C_{\mu,\gamma}:H^p\longrightarrow H^q$
is bounded if and only if $\mu$ is an $s$-Carleson measure. Moreover,
\begin{equation}\label{eq:upper-norm-equivalence}
 \norm{\mathcal C_{\mu,\gamma}}_{H^p\to H^q}
 \asymp
 \sup_{0\le t<1}\frac{\mu([t,1))}{(1-t)^s},
\end{equation}
where the comparison constants depend only on $p$, $q$ and $\gamma$.
\end{theorem}

The proof treats $0<p\le1$ and $1<p<q<\infty$
by different sufficiency arguments, but uses the same necessity argument in both ranges.
In the classical case $\gamma=1$, Theorem \ref{thm:upper-improved}, together with the known diagonal theorem  \cite[Theorem 1]{GGM2022}, implies that for $0< p\le q<\infty$,
\begin{equation*}\label{eq:upper-classic}
 \mathcal C_\mu:H^p\to H^q\text{ is bounded}
 \quad\Longleftrightarrow\quad
 \mu\text{ is a }\left(1+\frac1p-\frac1q\right)\text{-Carleson measure}.
\end{equation*}

Our second main theorem treats the complementary region $0< q<p\le\infty$.  The  normalized tail
\[
 t\longmapsto
 \frac{\mu([t,1))}{(1-t)^{\gamma+1/p-1/q}}
\]
must instead belong to $L^r((0,1),dt/(1-t))$, where
$1/r=1/q-1/p$. This criterion includes $p=\infty$ by taking $r=q$,
and has equivalent generating-function, boundary-potential, and moment
formulations.
Write
\begin{equation}\label{eq:Fmugamma}
 F_{\mu,\gamma}(z)
 :=\mathcal C_{\mu,\gamma}(1)(z)
 =\int_0^1\frac{d\mu(t)}{(1-tz)^\gamma}
 =\sum_{n=0}^\infty A_n^\gamma\mu_nz^n.
\end{equation}
If $\gamma=1$, we write $F_{\mu,1}=F_{\mu}$.

\begin{theorem}\label{thm:lower-main}
Let $0< q<p\le\infty$, let $\gamma>0$, and let $\mu$ be a finite
positive Borel measure on $[0,1)$. Define
\begin{equation*}\label{eq:r-sigma}
 \frac1r=\frac1q-\frac1p,
 \qquad
 \sigma=\gamma-\frac1r=\gamma+\frac1p-\frac1q,
 \qquad \frac1\infty=0.
\end{equation*}
Thus $0<r<\infty$, $r=pq/(p-q)$ for $p<\infty$, and $r=q$ for
$p=\infty$. For $j\ge0$ put
\[
 a_j=1-2^{-j},\qquad m_j=\mu([a_j,1)),\qquad d_j=2^{j\sigma}m_j.
\]
Then the following assertions are equivalent:
\begin{enumerate}[label=\textup{(\roman*)}]
 \item $\mathcal C_{\mu,\gamma}:H^p\to H^q$ is bounded;
 \item $\mathcal C_{\mu,\gamma}:H^p\to H^q$ is compact;
 \item $\{d_j\}_{j\ge0}\in\ell^r$;
 \item
 \begin{equation}\label{eq:tail-integral-main}
  \int_0^1
  \left(\frac{\mu([t,1))}{(1-t)^\sigma}\right)^r
  \frac{dt}{1-t}<\infty;
 \end{equation}
 \item the boundary potential
 \begin{equation*}\label{eq:P-mugamma}
  P_{\mu,\gamma}(e^{i\theta})
  :=\int_0^1\frac{d\mu(t)}{|1-te^{i\theta}|^\gamma}
 \end{equation*}
 belongs to $L^r(\T)$;
 \item $F_{\mu,\gamma}\in H^r$;
 \item
 \begin{equation*}\label{eq:lower-moment}
  \sum_{n=0}^\infty(n+1)^{r\gamma-2}\mu_n^r<\infty.
 \end{equation*}
\end{enumerate}
Moreover,
\begin{align}\label{eq:norm-equivalence-lower}
 \norm{\mathcal C_{\mu,\gamma}}_{H^p\to H^q}
 &\asymp\norm{F_{\mu,\gamma}}_{H^r}
 \asymp\norm{P_{\mu,\gamma}}_{L^r(\T)}\notag\\
 &\asymp\left(\sum_{j=0}^\infty d_j^r\right)^{1/r}
 \asymp\left(\sum_{n=0}^\infty(n+1)^{r\gamma-2}\mu_n^r\right)^{1/r}\notag\\
 &\asymp
 \left[\mu([0,1))^r+
 \int_0^1
 \left(\frac{\mu([t,1))}{(1-t)^\sigma}\right)^r
 \frac{dt}{1-t}\right]^{1/r}.
\end{align}
The comparison constants depend only on $p,q,\gamma$.
\end{theorem}

At $p=\infty$ we have $r=q$, and \eqref{eq:tail-integral-main} is precisely
\begin{equation}\label{eq:Hinfty-tail}
 \int_0^1\frac{\mu([t,1))^q}{(1-t)^{q\gamma}}\,dt<\infty.
\end{equation}
Thus the theorem includes the endpoint $H^\infty\to H^q$ for every
$1\le q<\infty$, with $F_{\mu,\gamma}\in H^q$ and the moment condition
$\sum_{n\ge0}(n+1)^{q\gamma-2}\mu_n^q<\infty$.

For the classical kernel $\gamma=1$ and $1\le q<p<\infty$,
Theorem~\ref{thm:lower-main} sharpens the sufficient condition in
\cite[Theorem~3.15(ii)]{Blasco2024}. That result implies boundedness under
\begin{equation}\label{eq:blasco-sufficient}
 \sum_{n=0}^\infty (n+1)^{-1/r}\mu_n<\infty,
\end{equation}
which, for positive measures, is equivalent to
\begin{equation}\label{eq:strong-integral}
 \int_0^1\frac{d\mu(t)}{(1-t)^{1-1/r}}<\infty.
\end{equation}
The exact condition  is strictly weaker; see Proposition \ref{prop:strictness} below.

Our third theorem treats the  endpoint $H^p\to H^\infty$
for $1<p<\infty$. In contrast to the range in
Theorem~\ref{thm:upper-improved}, the normalized tail must satisfy a
summability condition. The associated generating function is
$F_{\mu,\gamma+1}$, rather than $F_{\mu,\gamma}$. In both the lower triangle and this 
endpoint, boundedness also implies compactness. 

\begin{theorem}\label{thm:target-infty}
Let $1<p<\infty$, $\gamma>0$, and let $\mu$ be a finite positive
Borel measure on $[0,1)$. Put
\[
 s=\gamma+\frac1p,\qquad
 a_j=1-2^{-j},\qquad b_j=2^{js}\mu([a_j,1)),\quad j\ge0.
\]
Then the following assertions are equivalent:
\begin{enumerate}[label=\textup{(\roman*)}]
 \item $\mathcal C_{\mu,\gamma}:H^p\to H^\infty$ is bounded;
 \item $\mathcal C_{\mu,\gamma}:H^p\to H^\infty$ is compact;
 \item $b=\{b_j\}_{j\ge0}\in\ell^{p'}$;
 \item
 \begin{equation}\label{eq:target-tail}
  \int_0^1\left(\frac{\mu([t,1))}{(1-t)^{\gamma+1}}\right)^{p'}dt
  <\infty;
 \end{equation}
 \item $F_{\mu,\gamma+1}\in H^{p'}$;
 \item
 \begin{equation*}\label{eq:target-moment}
  \sum_{n=0}^\infty(n+1)^{p'(\gamma+1)-2}\mu_n^{p'}<\infty.
 \end{equation*}
\end{enumerate}
Moreover,
\begin{align}\label{eq:target-norm}
 \|\mathcal C_{\mu,\gamma}\|_{H^p\to H^\infty}
 &\asymp \|F_{\mu,\gamma+1}\|_{H^{p'}}
 \asymp \|P_{\mu,\gamma+1}\|_{L^{p'}(\T)}\notag\\
 &\asymp \|b\|_{\ell^{p'}}
 \asymp\left(\sum_{n=0}^\infty
      (n+1)^{p'(\gamma+1)-2}\mu_n^{p'}\right)^{1/p'}\notag\\
 &\asymp\left[\mu([0,1))^{p'}+
   \int_0^1\left(\frac{\mu([t,1))}{(1-t)^{\gamma+1}}\right)^{p'}dt
   \right]^{1/p'}.
\end{align}
The comparison constants depend only on $p$ and $\gamma$.
\end{theorem}

Since $p's+1=p'(\gamma+1)$, the integral in
\eqref{eq:target-tail} is exactly
\[
 \int_0^1\left(\frac{\mu([t,1))}{(1-t)^s}\right)^{p'}
 \frac{dt}{1-t}.
\]
Thus the $s$-Carleson condition alone does not suffice at this
endpoint; Example~\ref{ex:target-carleson-fails} makes the distinction
explicit. The proof of Theorem~\ref{thm:target-infty} will reuse the
measure comparisons already established for Theorem~\ref{thm:lower-main}.

For $\gamma=1$, the off-diagonal criteria combine with the known
diagonal and $H^\infty\to H^\infty$ results to give the full boundedness
classification for $1\le p,q\le\infty$.

\begin{corollary}\label{thm:complete-classical}
Let $0< p,q\le\infty$
and let $\mu$ be a finite positive Borel measure on $[0,1)$. Then the following statements hold.
\begin{enumerate}[label=\textup{(\alph*)}]
 \item if $0< q<\infty$, then
 $\mathcal C_\mu:H^p\to H^q$ is bounded if and only if $\mu$ is a
 $\left(1+1/p-1/q\right)$-Carleson measure;
 \item if $0< q<p\le\infty$ and $1/r=1/q-1/p$, then
 $\mathcal C_\mu:H^p\to H^q$ is bounded if and only if
 \[
  \sum_{n=0}^\infty(n+1)^{r-2}\mu_n^r<\infty;
 \]
 \item if $0<p\leq 1$ and $q=\infty$, then
 $\mathcal C_\mu:H^p\to H^\infty$ is bounded if and only if
 $\mu$ is a $(1+\frac{1}{p})$-Carleson measure;
 \item if $1<p\leq\infty$ and $q=\infty$, then
 $\mathcal C_\mu:H^p\to H^\infty$ is bounded if and only if
 \[
  \sum_{n=0}^\infty(n+1)^{2p'-2}\mu_n^{p'}<\infty.
 \]
\end{enumerate}
In cases \textup{(b)} and \textup{(d)}, boundedness is equivalent
to compactness.
\end{corollary}

The proof of Theorem~\ref{thm:upper-improved} uses normalized kernel tests for necessity.
Sufficiency follows from tail comparison and Hardy--Littlewood
mixed-mean estimates when $0<p\le1$, and from coefficient factorization,
Marcinkiewicz multipliers, and fractional integration when $p>1$.
For Theorem~\ref{thm:lower-main}, boundary-potential estimates and measure truncation
give boundedness and compactness, while dyadic kernel sums and radial
sampling reduce necessity to a diagonal multiplier criterion.
At $p=\infty$, testing with the constant function $1$ suffices
for necessity. Discrete convolution estimates identify the
equivalent tail and moment conditions throughout the lower triangle. Theorem~\ref{thm:target-infty} reuses these comparisons with kernel
order $\gamma+1$: Hadamard factorization and polynomial approximation
give boundedness and compactness, whereas positive kernel tests and
Fatou's lemma yield the necessary $\ell^{p'}$ condition.

The paper is organized as follows. Section~\ref{sec:preliminaries}
collects the preliminary results. Section~\ref{sec:upper} proves
Theorem~\ref{thm:upper-improved}, and Section~\ref{sec:lower} proves
Theorem~\ref{thm:lower-main}, including the source endpoint $p=\infty$.
Section~\ref{sec:target-infty} proves Theorem~\ref{thm:target-infty}
for the target $H^\infty$, reusing the measure comparisons from
Section~\ref{sec:lower}. Section~\ref{sec:classical} discusses the strictness of the previously known sufficient
condition and threshold examples, and Section~\ref{sec:conclusion}
contains concluding remarks.

\section{Preliminaries}\label{sec:preliminaries}

For $\beta>0$,  the Riemann--Liouville type fractional derivative is  defined by
\[
 R^\beta f(z)
 =\sum_{n=0}^\infty
 \frac{\Gamma(n+1+\beta)}{\Gamma(n+1)\Gamma(\beta+1)}a_nz^n,
 \qquad f(z)=\sum_{n=0}^\infty a_nz^n\in\Hol(\D).
\]
In particular, the coefficient of $z^n$ in $R^\gamma F_\mu$ is
$A_n^{\gamma+1}\mu_n$. This normalization is used in the factorization
in Subsection~\ref{subsec:upper-Banach}.

Let $X $ and $Y $ be two spaces of analytic functions on the unit disc $\D$. Let  $f(z)=\sum_{n=0}^\infty a_{n}z^{n}\in X$  and $\lambda:=\{\lambda_{n}\}_{n=0}^{\infty}$ be a sequence.
We can define the  multiplier operator  $T_{\lambda}$ as follows,
\begin{equation*}(T_{\lambda}f)(z)=\sum_{n=0}^\infty \lambda_{n}a_{n}z^{n}.\end{equation*}
If $T_{\lambda}:X\rightarrow Y$, then $\lambda$
is said to be a coefficient multiplier or simply multiplier
from $X $ into $Y$.   The   multipliers are closely related to the Hadamard product.

Recall that for analytic functions $u(z)=\sum_{n\ge0}u_nz^n$ and
$v(z)=\sum_{n\ge0}v_nz^n$, their Hadamard product is
\[
 (u\star v)(z)=\sum_{n\ge0}u_nv_nz^n.
\]

We shall use two classical multiplier facts. The following result is the classical Hardy--Littlewood fractional integration theorem for analytic Hardy spaces; see \cite{HardyLittlewood1932,Duren1970}.

\begin{lemma}\label{lem:HL-fractional}
Let $1<p<q<\infty$ and put
\[
 \delta=\frac1p-\frac1q>0.
\]
Then the fractional integration multiplier
\[
 J_\delta f(z)=\sum_{n=0}^\infty (n+1)^{-\delta}a_nz^n,
 \qquad f(z)=\sum_{n=0}^\infty a_nz^n,
\]
defines a bounded operator $J_\delta:H^p\to H^q$.
\end{lemma}

 The classical Marcinkiewicz multiplier theorem therefore gives a bounded Fourier multiplier on $L^p(\T)$.  Since the multiplier preserves nonnegative Fourier frequencies, its restriction to the analytic subspace is bounded on $H^p$; see, for example, \cite[Theorems~6.2.2 and~4.3.7]{Grafakos2014}.

\begin{lemma}\label{lem:marcinkiewicz}
Let $1<p<\infty$ and let $b=\{b_n\}_{n\ge0}$ satisfy
\begin{equation}\label{eq:marcinkiewicz-condition}
 \sup_{N\ge1}
 \left(
 |b_{2N-1}|+
 \sum_{n=N}^{2N-2}|b_{n+1}-b_n|
 \right)<\infty.
\end{equation}
Then $T_b$ is bounded on $H^p$, and its operator norm is controlled by the quantity in \eqref{eq:marcinkiewicz-condition} together with $|b_0|$.
\end{lemma}

The following  characterization is well known. The proof is simple and will be omitted.
\begin{lemma}\label{lem:moment-decay}
Let $s>0$ and let $\mu$ be a finite positive Borel measure on $[0,1)$.
Then $\mu$ is an $s$-Carleson measure if and only if
\[
 \mathfrak M_s(\mu):=\sup_{n\ge0}(n+1)^s\mu_n<\infty.
\]
Moreover,
\[
 \mathfrak C_s(\mu):=
 \sup_{0\le t<1}\frac{\mu([t,1))}{(1-t)^s}
 \asymp\mathfrak M_s(\mu).
\]
\end{lemma}

We record the discrete convolution inequality used in the potential,
kernel-synthesis, and moment estimates below.

\begin{lemma}\label{lem:discrete-young}
Define $(b*x)_m=\sum_{k\in\mathbb Z}b_{m-k}x_k$.
\begin{enumerate}[label=\textup{(\alph*)},leftmargin=*]
\item If $1\le u\le\infty$, $b\in\ell^1(\mathbb Z)$, and
$x\in\ell^u(\mathbb Z)$, then
\begin{equation*}\label{eq:discrete-young}
 \|b*x\|_{\ell^u}\le\|b\|_{\ell^1}\|x\|_{\ell^u}.
\end{equation*}
\item If $0<u<1$ and $b,x\in\ell^u(\mathbb Z)$, then
\begin{equation*}\label{eq:quasi-convolution}
 \|b*x\|_{\ell^u}^u
 \le\|b\|_{\ell^u}^u\|x\|_{\ell^u}^u.
\end{equation*}
\end{enumerate}
In both cases the defining series are absolutely convergent.
Sequences on $\mathbb N_0$ are extended by zero to negative indices.
\end{lemma}
\begin{proof}
Part \textup{(a)} is Young's inequality on $\mathbb Z$ with counting
measure; see \cite[Theorem~1.2.12]{Grafakos2014}.
For \textup{(b)}, recall that
$(\sum_k y_k)^u\le\sum_k y_k^u$ for $y_k\ge0$. Hence
$\ell^u\subset\ell^1\cap\ell^\infty$, so every defining convolution
series is absolutely convergent. Subadditivity and Tonelli's theorem
then give
\[
 \sum_m|(b*x)_m|^u
 \le\sum_m\sum_k|b_{m-k}|^u|x_k|^u
 =\left(\sum_\ell|b_\ell|^u\right)\sum_k|x_k|^u.
\]
This proves the assertion.
\end{proof}

The next lemma identifies the exact boundary potential that occurs in the lower triangle.

\begin{lemma}\label{lem:potential}
Let $r>0$, $\gamma>0$, and put $\sigma=\gamma-1/r$. For
\[
 a_j=1-2^{-j},\qquad m_j=\mu([a_j,1)),\qquad d_j=2^{j\sigma}m_j,
\]
one has
\begin{equation}\label{eq:potential-dyadic-equivalence}
 \norm{P_{\mu,\gamma}}_{L^r(\T)}^r
 \asymp
 \sum_{j=0}^\infty d_j^r.
\end{equation}
Furthermore,
\begin{equation}\label{eq:integral-dyadic-equivalence}
 \sum_{j=0}^\infty d_j^r
 \asymp
 \mu([0,1))^r+
 \int_0^1
 \left(\frac{\mu([t,1))}{(1-t)^\sigma}\right)^r\frac{dt}{1-t}.
\end{equation}
Consequently,  the integral in \eqref{eq:integral-dyadic-equivalence} is finite if and only if $\{d_j\}\in\ell^r$.
\end{lemma}

\begin{proof}
For $|\theta|\le \pi$ and $0\le t<1$, we have
\begin{equation}\label{eq:kernel-geometry}
	|1-te^{i\theta}|\asymp (1-t)+|\theta|,
\end{equation}
where the comparison constants are independent of $t$ and $\theta$.
Indeed, since
\[
1-te^{i\theta}=(1-t)+t(1-e^{i\theta})
\quad\text{and}\quad
|1-e^{i\theta}|=2\sin\frac{|\theta|}{2}\le|\theta|,
\]
the triangle inequality gives
\[
|1-te^{i\theta}|
\le (1-t)+t|1-e^{i\theta}|
\le (1-t)+|\theta|.
\]
For the reverse estimate, it cleat that
$
1-t\le |1-te^{i\theta}|.
$
Moreover, the elementary inequality
$\sin x\ge 2x/\pi$ for $0\le x\le\pi/2$ implies
\[
\frac{2}{\pi}|\theta|
\le |1-e^{i\theta}| \le |1-te^{i\theta}|+(1-t)
\le 2|1-te^{i\theta}|.
\]
This proves \eqref{eq:kernel-geometry}.

Let
\[
 E_k=\{\theta:2^{-k-1}<|\theta|\le2^{-k}\},\qquad k\ge1.
\]
If $\theta\in E_k$ and $t\in[a_k,1)$, then \eqref{eq:kernel-geometry} gives
$|1-te^{i\theta}|\lesssim2^{-k}$, and hence
\[
 P_{\mu,\gamma}(e^{i\theta})\gtrsim2^{k\gamma}m_k.
\]
Therefore
\[
 \int_{E_k}P_{\mu,\gamma}(e^{i\theta})^r\,d\theta
 \gtrsim2^{-k}2^{k\gamma r}m_k^r=d_k^r.
\]
Moreover, $|1-te^{i\theta}|\le2$ for all $t$ and $\theta$, so
$P_{\mu,\gamma}(e^{i\theta})\ge2^{-\gamma}m_0$. Thus
\[
 \norm{P_{\mu,\gamma}}_{L^r}^r\gtrsim\sum_{k=0}^\infty d_k^r.
\]

For the reverse estimate, fix $k\ge1$ and $\theta\in E_k$. Since $a_0=0$, we have
the disjoint decomposition
\[
 [0,1)
 =
 \left(\bigcup_{j=0}^{k-1}[a_j,a_{j+1})\right)
 \cup[a_k,1).
\]
Consequently,
\begin{align*}
 P_{\mu,\gamma}(e^{i\theta})
 &=
 \sum_{j=0}^{k-1}
 \int_{[a_j,a_{j+1})}
 \frac{d\mu(t)}{|1-te^{i\theta}|^\gamma}
 +
 \int_{[a_k,1)}
 \frac{d\mu(t)}{|1-te^{i\theta}|^\gamma}.
\end{align*}
If $t\in[a_j,a_{j+1})$, then
$
 2^{-j-1}<1-t\le2^{-j}.
$
This gives
\[
 |1-te^{i\theta}|\ge1-t>2^{-j-1},
\]
and hence
\[
 \int_{[a_j,a_{j+1})}
 \frac{d\mu(t)}{|1-te^{i\theta}|^\gamma}
\lesssim 2^{j\gamma}\mu([a_j,a_{j+1})).
\]
For the remaining interval $[a_k,1)$, by \eqref{eq:kernel-geometry} and the fact that
$\theta\in E_k$,
\[
 |1-te^{i\theta}|
 \gtrsim (1-t)+|\theta|
 \ge|\theta|>2^{-k-1}.
\]
Therefore,
\[
 \int_{[a_k,1)}
 \frac{d\mu(t)}{|1-te^{i\theta}|^\gamma}
 \lesssim
 2^{k\gamma}\mu([a_k,1))
 =
 2^{k\gamma}m_k.
\]
Combining these estimates and using
\[
 \mu([a_j,a_{j+1}))=m_j-m_{j+1}\le m_j,
\]
we obtain
\begin{align*}
 P_{\mu,\gamma}(e^{i\theta})
 \lesssim
 \sum_{j=0}^{k-1}
 2^{j\gamma}\mu([a_j,a_{j+1}))
 +2^{k\gamma}m_k
 \le
 \sum_{j=0}^{k}2^{j\gamma}m_j.
\end{align*}
The preceding pointwise estimate gives
\[
 P_{\mu,\gamma}(e^{i\theta})
\lesssim 2^{k/r}\sum_{j=0}^{k}2^{-(k-j)/r}d_j,
 \qquad \theta\in E_k.
\]
Extend $d$ by zero to the negative integers and define the one-sided
kernel on $\mathbb Z$ by
\[
 b^{(r)}_\ell=
 \begin{cases}
 2^{-\ell/r},&\ell\ge0,\\
 0,&\ell<0.
 \end{cases}
\]
Then
\[
 \norm{b^{(r)}}_{\ell^1(\mathbb Z)}
 =\frac{1}{1-2^{-1/r}}<\infty,
 \qquad
 (b^{(r)}*d)_k=\sum_{j=0}^k2^{-(k-j)/r}d_j
 \quad(k\ge0).
\]
For $k\ge1$, using $|E_k|=2^{-k}$ we have
\[
 \int_{E_k}P_{\mu,\gamma}(e^{i\theta})^r\,d\theta
\lesssim 2^k|E_k|\,|(b^{(r)}*d)_k|^r=
|(b^{(r)}*d)_k|^r.
\] 
If $\theta \in A:=\{\theta\in[-\pi,\pi]:|\theta|>1/2\}$, the earlier estimate \eqref{eq:kernel-geometry} 
and $(b^{(r)}*d)_0=d_0=m_0$ give
\[
 \int_A P_{\mu,\gamma}(e^{i\theta})^r\,d\theta
 \lesssim|(b^{(r)}*d)_0|^r.
\]

If $r\ge1$, use Lemma~\ref{lem:discrete-young}\textup{(a)} and
$\|b^{(r)}\|_{\ell^1}=(1-2^{-1/r})^{-1}$. If $0<r<1$, use
part \textup{(b)} instead, since
\[
 \|b^{(r)}\|_{\ell^r}^r=\sum_{\ell\ge0}2^{-\ell}=2.
\]
In either case,  we  have
\begin{align*}
 \norm{P_{\mu,\gamma}}_{L^r(\T)}^r
 \lesssim\sum_{k\ge0}d_k^r.
\end{align*}
Together with the lower estimate, this proves
\eqref{eq:potential-dyadic-equivalence}.

For every $j\ge0$ and $t\in[a_j,a_{j+1})$, one has
$1-t\asymp2^{-j}$ and $m_{j+1}\le\mu([t,1))\le m_j$. Consequently,
\[
 2^{j\sigma r}m_{j+1}^r
 \lesssim
 \int_{a_j}^{a_{j+1}}
 \left(\frac{\mu([t,1))}{(1-t)^\sigma}\right)^r\frac{dt}{1-t}
 \lesssim 2^{j\sigma r}m_j^r.
\]
This gives
\[
 2^{-\sigma r}\sum_{j=1}^\infty d_j^r
 \lesssim  \int_{0}^{1}\left(\frac{\mu([t,1))}{(1-t)^\sigma}\right)^r\frac{dt}{1-t}\lesssim\sum_{j=0}^\infty d_j^r.
\]
Since $d_0^r=m_0^r=\mu([0,1))^r$, adding this term yields
\[
 m_0^r+\int_{0}^{1}\left(\frac{\mu([t,1))}{(1-t)^\sigma}\right)^r\frac{dt}{1-t}\asymp\sum_{j=0}^\infty d_j^r,
\]
which proves \eqref{eq:integral-dyadic-equivalence}. The mass term is
needed because an atom at $0$ is not detected by the integral over
$(0,1)$.
\end{proof}

The next lemma supplies the test functions used for necessity.

\begin{lemma}\label{lem:synthesis}
Let $0< p<\infty$, and put $a_j=1-2^{-j}$ for $j\ge0$.
Define
\begin{equation}\label{eq:kj}
 k_j(z)=\frac{(1-a_j^2)^{1/p}}{(1-a_jz)^{2/p}},
 \qquad z\in\mathbb D.
\end{equation}
Then $\|k_j\|_{H^p}=1$ for every $j\ge0$.
Moreover, for every complex sequence
$c=\{c_j\}_{j\ge0}\in\ell^p$, the series
\[
 F(z):=\sum_{j=0}^{\infty}c_jk_j(z)
\]
converges in $H^p$ and absolutely and uniformly on compact
subsets of $\mathbb D$. There exists a constant $C_p>0$,
depending only on $p$, such that
\begin{equation}\label{eq:synthesis-est}
 \|F\|_{H^p}
 \le C_p\left(\sum_{j=0}^{\infty}|c_j|^p\right)^{1/p}.
\end{equation}
\end{lemma}

\begin{proof}

The equality $\|k_j\|_{H^p}=1$ follows from the Poisson kernel
identity. Suppose first that $c$ is a finitely supported
complex sequence and put
$
 F=\sum_{j\ge0}c_jk_j.
$
For $0<p\le1$, subadditivity gives directly
\begin{equation}\label{eq:synthesis-quasi-finite}
 \|F\|_{H^p}^p\le\sum_j|c_j|^p\|k_j\|_{H^p}^p
 =\sum_j|c_j|^p.
\end{equation}

For $p>1$, use the arcs $E_m,A$ from Lemma~\ref{lem:potential}. Since $1-a_j^2\asymp2^{-j}$, \eqref{eq:kernel-geometry} yields,
\[
 |k_j(e^{i\theta})|
 \lesssim
 \begin{cases}
 2^{j/p},& j\le m,\\
 2^{(2m-j)/p},& j>m.
 \end{cases}
\]
Consequently,
\begin{align*}
 2^{-m/p}\sum_j|c_j||k_j(e^{i\theta})|
 &\lesssim
 \sum_{j\le m}2^{-(m-j)/p}|c_j|
 +\sum_{j>m}2^{-(j-m)/p}|c_j|\\
 &\le \sum_j2^{-|m-j|/p}|c_j|.
\end{align*}
Extend $c$ by zero to the negative integers and define
\[
 v^{(p)}_\ell=2^{-|\ell|/p},\qquad \ell\in\mathbb Z.
\]
Then
\[
 \norm{v^{(p)}}_{\ell^1(\mathbb Z)}
 =1+2\sum_{\ell=1}^\infty2^{-\ell/p}
 =\frac{1+2^{-1/p}}{1-2^{-1/p}}<\infty,
\]
and
\[
 (v^{(p)}*|c|)_m=\sum_{j\ge0}2^{-|m-j|/p}|c_j|.
\]
As $|E_m|=2^{-m}$, the preceding estimate gives
\[
 \int_{E_m}\left|F(e^{i\theta})\right|^p\,d\theta
\lesssim2^m|E_m|\,|(v^{(p)}*|c|)_m|^p
 =|(v^{(p)}*|c|)_m|^p.
\]

On the fixed arc $A$, the lower bound proved in
Lemma~\ref{lem:potential}  also gives
\[
 \int_A \left|F(e^{i\theta})\right|^p\,d\theta
 \lesssim |(v^{(p)}*|c|)_0|^p.
\]
Combining the estimates on $A$ and the $E_m$, and applying
Lemma~\ref{lem:discrete-young}\textup{(a)} with $u=p$, gives
\begin{align*}
 \norm{F}_{H^p}^p
 &\lesssim\sum_{m\ge0}|(v^{(p)}*|c|)_m|^p\\
 &\le \norm{v^{(p)}}_{\ell^1(\mathbb Z)}^p
             \norm{c}_{\ell^p(\mathbb Z)^p}\\
 &\lesssim \sum_{j\ge0}|c_j|^p.
\end{align*}
Taking $p$-th roots proves \eqref{eq:synthesis-est}.

For an arbitrary sequence $c\in\ell^p$, put
\[
 F_N=\sum_{j=0}^{N}c_jk_j.
\]
The estimate for finitely supported sequences gives, for $M>N$,
\[
 \|F_M-F_N\|_{H^p}
 \le C_p
 \left(\sum_{j=N+1}^{M}|c_j|^p\right)^{1/p}.
\]
Thus $\{F_N\}$ is Cauchy in $H^p$ and converges to some
$F\in H^p$. Passing to the limit yields
\[
 \|F\|_{H^p}\le C_p\|c\|_{\ell^p},
 \qquad
 \|F-F_N\|_{H^p}
 \le C_p\left(\sum_{j>N}|c_j|^p\right)^{1/p}.
\]

Fix $0<\rho<1$. For $|z|\le\rho$, we have
\[
 |1-a_jz|\ge1-a_j|z|\ge1-\rho,
\]
and hence
\[
 \sup_{|z|\le\rho}|k_j(z)|
 \le\frac{2^{1/p}}{(1-\rho)^{2/p}}\,2^{-j/p}.
\]
Since $\|c\|_{\ell^\infty}\le\|c\|_{\ell^p}$,
\begin{align*}
 \sum_{j=0}^{\infty}\sup_{|z|\le\rho}|c_jk_j(z)|
 &\le\frac{2^{1/p}}{(1-\rho)^{2/p}}
       \sum_{j=0}^{\infty}|c_j|2^{-j/p}\\
 & \lesssim \frac{\|c\|_{\ell^p}}
 {(1-\rho)^{2/p}}<\infty.
\end{align*}
Thus the series $\sum_{j\ge0}c_jk_j$ converges absolutely and
uniformly on every compact subset of $\mathbb D$. Its sum agrees with $F$. In fact, the  point-evaluation
estimate gives
\[
 \sup_{|z|\le\rho}|F_N(z)-F(z)|
 \le C_{p,\rho}\|F_N-F\|_{H^p}\longrightarrow0.
\]
Consequently, $F=\sum_{j\ge0}c_jk_j$.
 This completes the proof.
\end{proof}

\begin{lemma}\label{lem:sampling}
For $a_j=1-2^{-j}$, the discrete measure
\[
 \nu=\sum_{j=0}^\infty(1-a_j)\delta_{a_j}
\]
is a Carleson measure on $\D$. Consequently, for every $0<q<\infty$,
\begin{equation*}\label{eq:sampling}
 \sum_{j=0}^\infty(1-a_j)|g(a_j)|^q
 \lesssim \norm{g}_{H^q}^q.
\end{equation*}
\end{lemma}

\begin{proof}
The radial tail satisfies
\[
 \sum_{a_j\ge x}(1-a_j)\lesssim1-x,
\]
so that  $\nu$ is  a Carleson measure. The estimate is the classical Carleson embedding theorem for Hardy spaces, see \cite{Duren1970}.
\end{proof}

For Banach exponents, the next result is a special
case of \cite[Proposition~8]{OrhonTerzioglu1972}; see also
\cite[Section~2]{Fischer2020}. We include the elementary argument
because the lower-triangle proof also needs exponents below $1$.

\begin{lemma}\label{lem:diagonal}
Let $0< q<p<\infty$ and $1/r=1/q-1/p$. For a nonnegative sequence $d=\{d_j\}$, the diagonal map
\[
 D_d:\ell^p\to\ell^q,\qquad D_d(c)=\{d_jc_j\},
\]
is bounded if and only if $d\in\ell^r$. Moreover,
\[
 \norm{D_d}_{\ell^p\to\ell^q}=\norm{d}_{\ell^r}.
\]
\end{lemma}

\begin{proof}
If $d\in\ell^r$, apply H\"older's inequality to
$\sum_jd_j^q|c_j|^q$ with exponents $r/q$ and $p/q$.
Both exceed $1$, and $q/r+q/p=1$, so
$\|D_dc\|_{\ell^q}\le\|d\|_{\ell^r}\|c\|_{\ell^p}$.
Conversely, suppose $\|D_dc\|_{\ell^q}\le L\|c\|_{\ell^p}$
for finite nonnegative sequences. On a finite set $J$, take
$c_j=d_j^{r/p}$ and put $c_j=0$ off $J$.
Since $q(1+r/p)=r$, this gives
\[
 \left(\sum_{j\in J}d_j^r\right)^{1/q}
 \le L\left(\sum_{j\in J}d_j^r\right)^{1/p}.
\]
Unless the sum is zero, divide to obtain
$(\sum_{j\in J}d_j^r)^{1/r}\le L$; the zero case is immediate.
Letting $J$ increase proves the lower norm bound.  Replacing a complex sequence by
its modulus and taking finite truncations justifies the last assertion.
\end{proof}

The following comparison applies to nondecreasing test functions. 

\begin{lemma}\label{lem:tail-comparison}
Let $s>0$, and let $\mu$ be a positive Borel measure on $[0,1)$ such that
\[
 \mu([a,1))\le K(1-a)^s,\qquad 0\le a<1,
\]
for some $K>0$. If $h:[0,1)\to[0,\infty]$ is a nonnegative,
nondecreasing Borel function, then
\begin{equation*}\label{eq:tail-comparison}
 \int_{[0,1)}h(t)\,d\mu(t)
 \le Ks\int_0^1h(t)(1-t)^{s-1}\,dt.
\end{equation*}
The integrals are understood in the extended nonnegative sense.
\end{lemma}

\begin{proof}
Let $d\nu_s(t)=s(1-t)^{s-1}\,dt$, so that
$\nu_s([a,1))=(1-a)^s$. For $\lambda\ge0$, define
$$E_\lambda=\{t\in[0,1):h(t)>\lambda\}.$$
If $E_\lambda$ is nonempty and $a_\lambda=\inf E_\lambda$,
then
\[
 (a_\lambda,1)\subseteq E_\lambda\subseteq[a_\lambda,1).
\]
Indeed, for $t>a_\lambda$ there is $u\in E_\lambda$ with $u<t$,
so $h(t)\ge h(u)>\lambda$. Since $\nu_s$ has no atoms,
\[
 \nu_s(E_\lambda)=(1-a_\lambda)^s,
 \qquad
 \mu(E_\lambda)\le\mu([a_\lambda,1))
 \le K(1-a_\lambda)^s=K\nu_s(E_\lambda).
\]
The same comparison is immediate for an empty superlevel set.
By Fubini's theorem  we obtain
\begin{align*}
 \int_{[0,1)}h(t)\,d\mu(t)
 &=\int_0^\infty\mu(E_\lambda)\,d\lambda\\
 &\le K\int_0^\infty\nu_s(E_\lambda)\,d\lambda\\
 &=K\int_{[0,1)}h(t)\,d\nu_s(t)
 =Ks\int_0^1h(t)(1-t)^{s-1}\,dt.
\end{align*}
\end{proof}

We shall use the following embedding, see \cite[Theorem~5.11]{Duren1970} and \cite{GirelaMarquez1998}.

\begin{lemma}\label{lem:HL-small-p}
Let $0<p\le1$, $1\le q\le\infty$, and $p<q$. Then
\begin{equation*}\label{eq:HL-small-p}
 \int_0^1(1-t)^{1/p-1/q-1}M_q(t,f)\,dt
 \lesssim \norm{f}_{H^p},\qquad f\in H^p.
\end{equation*}
\end{lemma}

\section{The upper triangle}\label{sec:upper}

We prove Theorem \ref{thm:upper-improved}. Necessity is common to all
indices in the statement. For $0<p\le1$, sufficiency follows from
Lemmas \ref{lem:tail-comparison} and \ref{lem:HL-small-p}. For $1<p<q<\infty$, we use
a coefficient factorization, the Marcinkiewicz multiplier theorem, and
Hardy--Littlewood fractional integration. For $0<p\leq q <1$,  we adapt the  argument in the proof 
of \cite[Theorem 2.1]{Stempak1994} to an  positive Carleson measure and distinct Hardy exponents. 

\subsection{Necessity for all admissible indices}\label{subsec:upper-necessity}\hfill\\\indent

\begin{proof}[Necessity in Theorem \ref{thm:upper-improved}]
Let
\[
 N=\norm{\mathcal C_{\mu,\gamma}}_{H^p\to H^q}<\infty,
 \qquad s=\gamma+\frac1p-\frac1q.
\]
For $0\le a<1$, consider
\[
 f_a(z)=\frac{(1-a^2)^{1/p}}{(1-az)^{2/p}}.
\]
Then it is easy to see that  $\norm{f_a}_{H^p}=1$ for every
$p>0$. The  point-evaluation estimate of Hardy space yields
\begin{equation}\label{eq:upper-evaluation}
 |\mathcal C_{\mu,\gamma}f_a(a)|
 \lesssim N(1-a)^{-1/q},
\end{equation}
where the right-hand factor is $1$ when $q=\infty$.
For $a\le t<1$,
\[
 1-ta^2\le1-a^3\le3(1-a),\qquad
 1-ta\le1-a^2\le2(1-a).
\]
Since $1-a^2\ge1-a$, this gives
\begin{align*}
 \mathcal C_{\mu,\gamma}f_a(a)
 &\ge\int_{[a,1)}
 \frac{(1-a^2)^{1/p}}{(1-ta^2)^{2/p}(1-ta)^\gamma}\,d\mu(t)\\
 &\gtrsim
 \frac{\mu([a,1))}{(1-a)^{\gamma+1/p}}.
\end{align*}
Combining this lower bound with \eqref{eq:upper-evaluation} gives
\[
 \mathfrak C_s(\mu)
 :=\sup_{0\le a<1}\frac{\mu([a,1))}{(1-a)^s}
 \lesssim N.
\]
Thus $\mu$ is an
$s$-Carleson measure, and the lower estimate in
\eqref{eq:upper-norm-equivalence} holds throughout the stated range.
\end{proof}

\subsection{Sufficiency for \texorpdfstring{$0<p\le1,1\le q\le\infty, p<q$}{0<p<=1}}\hfill\\\indent

\begin{proof}
Assume $0<p\le1$, $1\le q\le\infty$, $p<q$, and
\[
 K=\mathfrak C_s(\mu)<\infty,
 \qquad s=\gamma+\frac1p-\frac1q>0.
\]
Fix $f\in H^p$ and $0<\rho<1$. By Minkowski's integral inequality  we have
\begin{align*}\label{eq:upper-small-p-minkowski}
 M_q(\rho,\mathcal C_{\mu,\gamma}f)
 &\le\int_{[0,1)}
 M_q\!\left(\rho,\frac{f(t\,\cdot)}{(1-t\,\cdot)^\gamma}\right)
 \,d\mu(t)\notag\\
 &\le\int_{[0,1)}
 \frac{M_q(\rho t,f)}{(1-\rho t)^\gamma}\,d\mu(t).
\end{align*}
For $q=\infty$, taking the supremum over the circle gives the same
inequality directly.

The function
\[
 h_\rho(t)=\frac{M_q(\rho t,f)}{(1-\rho t)^\gamma}
\]
is nonnegative, nondecreasing, and bounded on $[0,1)$. Since
$M_q(\rho t,f)\le M_q(t,f)$ and $1-\rho t\ge1-t$,   Lemma
\ref{lem:tail-comparison}  shows that 
\begin{align*}
 M_q(\rho,\mathcal C_{\mu,\gamma}f)
 & \lesssim \int_{[0,1)}
 \frac{M_q(\rho t,f)}{(1-\rho t)^\gamma}\,d\mu(t)\\
 &\le Ks\int_0^1
 \frac{M_q(\rho t,f)}{(1-\rho t)^\gamma}(1-t)^{s-1}\,dt\\
 &\le Ks\int_0^1M_q(t,f)(1-t)^{s-\gamma-1}\,dt\\
 &=Ks\int_0^1M_q(t,f)(1-t)^{1/p-1/q-1}\,dt\\
 &\lesssim K\norm{f}_{H^p}.
\end{align*}
The last estimate is Lemma \ref{lem:HL-small-p}. Taking the supremum
over $\rho$ we have 
\[
 \norm{\mathcal C_{\mu,\gamma}f}_{H^q}
 \lesssim \mathfrak C_s(\mu)\norm{f}_{H^p}.
\]\end{proof}
This completes the proof.
\subsection{Sufficiency for \texorpdfstring{$1<p<q<\infty$}{1<p<q<infinity}}\label{subsec:upper-Banach}\hfill\\\indent

Recall that $A_n^\eta=\Gamma(n+\eta)/(\Gamma(\eta)\Gamma(n+1))$. The Ces\`aro averaging operator of order $\gamma$ is
\begin{equation*}\label{eq:Cgamma-def}
 C^\gamma f(z)
 =\sum_{n=0}^\infty\frac1{A_n^{\gamma+1}}
 \left(\sum_{k=0}^nA_{n-k}^\gamma a_k\right)z^n,
 \qquad f(z)=\sum_{k=0}^\infty a_kz^k.
\end{equation*}

Andersen proved that
\begin{equation}\label{eq:Cgamma-Hp}
 C^\gamma:H^p\longrightarrow H^p
\end{equation}
is bounded for every $0<p<\infty$ and every $\gamma>0$; see \cite{Andersen1996}.
Moreover, with $F_\mu$ as in \eqref{eq:Fmugamma},
\begin{equation}\label{eq:C-factorization}
 \mathcal C_{\mu,\gamma}f=(R^\gamma F_\mu)\star C^\gamma f.
\end{equation}
Indeed, the $n$-th coefficient of $R^\gamma F_\mu$ is $A_n^{\gamma+1}\mu_n$, so the factors $A_n^{\gamma+1}$ cancel in the Hadamard product.

\begin{proof}[Sufficiency in Theorem \ref{thm:upper-improved} when $1<p<q<\infty$]
Set
\[
 \delta=\frac1p-\frac1q>0,
 \qquad s=\gamma+\delta.
\]
Suppose that $\mu$ is an $s$-Carleson measure and let
\[
 \mathfrak C_s(\mu)=
 \sup_{0\le t<1}\frac{\mu([t,1))}{(1-t)^s}.
\]
Define
\[
 \lambda_n=A_n^{\gamma+1}\mu_n,
 \qquad
 w_n=(n+1)^\delta A_n^{\gamma+1},
 \qquad
 b_n=w_n\mu_n.
\]
Then
\begin{equation}\label{eq:lambda-b-factor}
 \lambda_n=(n+1)^{-\delta}b_n,
 \qquad T_\lambda=J_\delta T_b.
\end{equation}
Since
$
 A_n^{\gamma+1}\asymp (n+1)^\gamma,
$
we have
\begin{equation}\label{eq:w-asymptotic}
 w_n\asymp(n+1)^s.
\end{equation}
Lemma \ref{lem:moment-decay} therefore yields
\begin{equation}\label{eq:b-bounded}
 \sup_{n\ge0}|b_n|\lesssim \mathfrak C_s(\mu).
\end{equation}

We next verify the dyadic bounded-variation condition.  The sequence $\{\mu_n\}$ is decreasing, while $\{w_n\}$ is increasing because
\[
 \frac{w_{n+1}}{w_n}
 =\left(\frac{n+2}{n+1}\right)^\delta
 \frac{n+\gamma+1}{n+1}>1.
\]
Hence, for $n\ge0$,
\begin{align*}
 |b_{n+1}-b_n|
 &=|w_{n+1}\mu_{n+1}-w_n\mu_n|\\
 &\le w_{n+1}(\mu_n-\mu_{n+1})
 +(w_{n+1}-w_n)\mu_n.
\end{align*}
Fix $N\ge1$.  Summing from $N$ to $2N-2$ gives
\begin{align*}
 \sum_{n=N}^{2N-2}|b_{n+1}-b_n|
 &\le
 w_{2N-1}(\mu_N-\mu_{2N-1})
 +\mu_N(w_{2N-1}-w_N)\\
 &\le 2w_{2N-1}\mu_N.
\end{align*}
Using \eqref{eq:w-asymptotic} and Lemma \ref{lem:moment-decay},
\[
 w_{2N-1}\mu_N\lesssim \mathfrak C_s(\mu).
\]
Together with \eqref{eq:b-bounded}, this gives
\[
 \sup_{N\ge1}
 \left(
 |b_{2N-1}|+
 \sum_{n=N}^{2N-2}|b_{n+1}-b_n|
 \right)
 \lesssim \mathfrak C_s(\mu).
\]
Lemma \ref{lem:marcinkiewicz} now implies
\begin{equation}\label{eq:Tb-upper}
 \norm{T_b}_{H^p\to H^p}
 \lesssim \mathfrak C_s(\mu).
\end{equation}
By Lemma \ref{lem:HL-fractional}, \eqref{eq:lambda-b-factor}, and \eqref{eq:Tb-upper},
\begin{equation*}\label{eq:Tlambda-upper}
 \norm{T_\lambda g}_{H^q}
 \lesssim \mathfrak C_s(\mu)\norm{g}_{H^p},
 \qquad g\in H^p.
\end{equation*}
Finally, \eqref{eq:C-factorization} and \eqref{eq:Cgamma-Hp} give
\[
 \norm{\mathcal C_{\mu,\gamma}f}_{H^q}
 =\norm{T_\lambda(C^\gamma f)}_{H^q}
 \lesssim \mathfrak C_s(\mu)\norm{f}_{H^p}.
\]
Thus $\mathcal C_{\mu,\gamma}:H^p\to H^q$ is bounded for every $\gamma>0$.
\end{proof}

\subsection{Sufficiency for \texorpdfstring{$0<p\leq q <1$}{1<p<q<infinity}}\hfill\\\indent

\begin{proof}[Sufficiency in Theorem \ref{thm:upper-improved} when $0<p\leq q <1$]
Let $\gamma>0$, put
\[
 K=\sup_{0\le t<1}\frac{\mu([t,1))}{(1-t)^s}<\infty,
\]
and fix $f\in H^p$. Define
\[
 h(z)=\frac{f(z)}{(1-z)^\gamma},
 \qquad
 a_k=1-2^{-k},\qquad I_k=[a_k,a_{k+1}),\quad k\ge0.
\]
The Carleson condition implies
\[
 \mu(I_k)\le\mu([a_k,1))\le K2^{-ks}.
\]

Fix $0<\rho<1$. Since $\mu$ is finite and $h$ is bounded on
$\{|z|\le\rho\}$, the following decomposition is absolutely
convergent:
\[
 \mathcal C_{\mu,\gamma}f(\rho e^{i\theta})
 =\sum_{k=0}^{\infty}\int_{I_k}h(\rho t e^{i\theta})\,d\mu(t).
\]
Set
\[
 H_{k,\rho}(\theta)
 =\sup_{0\le t\le a_{k+1}}|h(\rho t e^{i\theta})|.
\]
Then 
\begin{align*}
 |\mathcal C_{\mu,\gamma}f(\rho e^{i\theta})|^q
 &\le\sum_{k\ge0}
 \left|\int_{I_k}h(\rho t e^{i\theta})\,d\mu(t)\right|^q\\
 &\le K^q\sum_{k\ge0}2^{-ksq}H_{k,\rho}(\theta)^q.
\end{align*}
For each $k$, the function $z\mapsto h(\rho a_{k+1}z)$
belongs to $H^\infty$. The radial maximal inequality, applied to
this dilation, yields
\[
 \int_{-\pi}^{\pi}H_{k,\rho}(\theta)^q\,\frac{d\theta}{2\pi}
 \lesssim M_q(\rho a_{k+1},h)^q
 \le M_q(a_{k+1},h)^q.
\]
Thus Tonelli's theorem gives
\begin{equation}\label{eq:stempak-dyadic-bound}
 M_q(\rho,\mathcal C_{\mu,\gamma}f)^q
 \lesssim K^q\sum_{k\ge0}2^{-ksq}M_q(a_{k+1},h)^q.
\end{equation}

Since $s q>0$,
\[
 \int_{a_{k+1}}^{a_{k+2}}(1-t)^{sq-1}\,dt
 =\frac{2^{-sq}-2^{-2sq}}{sq}\,2^{-ksq}\asymp 2^{-ksq}.
\]
The monotonicity of $M_q(t,h)$ therefore implies
\[
 2^{-ksq}M_q(a_{k+1},h)^q
 \lesssim
 \int_{a_{k+1}}^{a_{k+2}}
       (1-t)^{sq-1}M_q(t,h)^q\,dt.
\]
Summing over the disjoint intervals and using
\eqref{eq:stempak-dyadic-bound}, we obtain
\begin{equation}\label{eq:stempak-weighted-bound}
 M_q(\rho,\mathcal C_{\mu,\gamma}f)^q
 \lesssim K^q
 \int_0^1(1-t)^{sq-1}M_q(t,h)^q\,dt.
\end{equation}

Choose $A>1$ sufficiently close to $1$ that
\[
 \max\{0,1-\gamma q\}<\frac1A<1,
\]
and let $B=A/(A-1)$ and $u=Aq$.(Note that the choice of $A$ is possible for every $\gamma>0$.) Then
\[
 \gamma qB>1,\qquad u>q\ge p.
\]
H\"older's inequality with exponents $A$ and $B$ gives
\begin{align*}
 M_q(t,h)^q
 &=\int_{-\pi}^{\pi}
   \frac{|f(te^{i\theta})|^q}{|1-te^{i\theta}|^{\gamma q}}
   \,\frac{d\theta}{2\pi}\\
 &\le M_u(t,f)^q
 \left(\int_{-\pi}^{\pi}
 |1-te^{i\theta}|^{-\gamma qB}\,
 \frac{d\theta}{2\pi}\right)^{1/B}\\
 &\lesssim 
 M_u(t,f)^q(1-t)^{-\gamma q+\frac{1}{B}}.
\end{align*}
The last step uses the elementary kernel estimate
\[
 \int_{-\pi}^{\pi}|1-te^{i\theta}|^{-\beta}
 \,\frac{d\theta}{2\pi}
 \lesssim (1-t)^{1-\beta},\qquad \beta>1,
\]
with $\beta=\gamma qB$.
Since $1/B=1-1/A$ and $q/u=1/A$, we have
\[
 sq-1-\gamma q+\frac1B
 =\frac qp-1-\frac1A
 =q\left(\frac1p-\frac1u\right)-1.
\]
Consequently, \eqref{eq:stempak-weighted-bound} becomes
\[
 M_q(\rho,\mathcal C_{\mu,\gamma}f)^q
 \lesssim K^q
 \int_0^1
 (1-t)^{q(1/p-1/u)-1}M_u(t,f)^q\,dt.
\]
The Hardy--Littlewood mixed-mean theorem states that, for
$0<p<u<\infty$ and $\lambda\ge p$,
\[
 \int_0^1
 (1-t)^{\lambda(1/p-1/u)-1}M_u(t,f)^\lambda\,dt
 \lesssim_{p,u,\lambda}\|f\|_{H^p}^\lambda;
\]
see \cite[Theorem~5.11]{Duren1970} and
\cite{GirelaMarquez1998}. We may apply it with $\lambda=q$
because $q\ge p$ and $u>q\ge p$. It follows that
\[
 M_q(\rho,\mathcal C_{\mu,\gamma}f)^q
 \lesssim K^q\|f\|_{H^p}^q.
\]
Taking the supremum over $0<\rho<1$  gives
\[
 \|\mathcal C_{\mu,\gamma}f\|_{H^q}
 \lesssim K\|f\|_{H^p}.
\]
This proves sufficiency throughout $0<p\le q<1$. 
\end{proof}

Together with Subsection \ref{subsec:upper-necessity}, the three
sufficiency arguments establish Theorem \ref{thm:upper-improved}
and its norm equivalence for all the stated indices.

\section{The lower triangle, including the \texorpdfstring{$H^\infty$}{H-infinity} endpoint}\label{sec:lower}

We prove Theorem~\ref{thm:lower-main} using the dyadic condition
\textup{(iii)} as the central criterion. The proof follows the route
\begin{gather*}
 \textup{(iii)}\Longleftrightarrow\textup{(iv)}
 \Longleftrightarrow\textup{(v)}\Longleftrightarrow\textup{(vi)}
 \Longleftrightarrow\textup{(vii)},\\
 \textup{(iii)}\Longrightarrow\textup{(ii)}
 \Longrightarrow\textup{(i)}\Longrightarrow\textup{(iii)}.
\end{gather*}
The first line concerns only the measure and the parameters $r,\gamma$.
Lemma~\ref{lem:potential} already treats \textup{(iii)--(v)};
Subsection~\ref{subsec:lower-measure-criteria} supplies the generating-function
and moment comparisons. The operator implications are then proved in
one argument in Subsection~\ref{subsec:lower-proof}. Only the necessity
of \textup{(iii)} requires separate treatments of finite $p$ and
$p=\infty$.

\subsection{Equivalent descriptions of the measure}\label{subsec:lower-measure-criteria}\hfill\\\indent

We first note a kernel estimate used below. For $0<\rho<1$,
$0\le t<1$, and $\theta\in\mathbb R$, since $t(1-\rho)\le1-\rho t\le|1-\rho te^{i\theta}|$, we obtain 
\begin{equation}\label{eq:radial-denominator}
 |1-te^{i\theta}|
 \le |1-\rho te^{i\theta}|+t(1-\rho)
 \le 2|1-\rho te^{i\theta}|,
\end{equation}

\begin{lemma}\label{lem:generating-tail}
Let $0< r<\infty$, $\gamma>0$, and let $\mu$ be a finite positive
Borel measure on $[0,1)$. Put
\[
 a_j=1-2^{-j},\qquad m_j=\mu([a_j,1)),\qquad
 d_j=2^{j(\gamma-1/r)}m_j.
\]
Then $F_{\mu,\gamma}\in H^r$ if and only if $d\in\ell^r$, and,
whenever these conditions hold,
\begin{equation*}\label{eq:general-F-potential}
 \|F_{\mu,\gamma}\|_{H^r}\asymp\|d\|_{\ell^r}.
\end{equation*}
\end{lemma}

\begin{proof}
If $d\in\ell^r$, Lemma~\ref{lem:potential} gives
$P_{\mu,\gamma}\in L^r(\T)$. By \eqref{eq:radial-denominator},
\[
 |F_{\mu,\gamma}(\rho e^{i\theta})|
 \le\int_{[0,1)}\frac{d\mu(t)}{|1-\rho te^{i\theta}|^\gamma}
 \le2^\gamma P_{\mu,\gamma}(e^{i\theta}).
\]
Taking integral means and then the supremum over $\rho$ yields
\[
 \|F_{\mu,\gamma}\|_{H^r}
 \le2^\gamma\|P_{\mu,\gamma}\|_{L^r(\T)}
 \lesssim\|d\|_{\ell^r}.
\]
Conversely, suppose $F_{\mu,\gamma}\in H^r$.
For $t\in[a_j,1)$, it obvious that
$1-ta_j\le1-a_j^2\le2(1-a_j)$. This gives
\[
 F_{\mu,\gamma}(a_j)
 \ge\int_{[a_j,1)}\frac{d\mu(t)}{(1-ta_j)^\gamma}
 \ge2^{-\gamma}2^{j\gamma}m_j.
\]
Applying Lemma~\ref{lem:sampling} with exponent $r>0$, we obtain
\[
 \sum_{j\ge0}d_j^r
 \le2^{\gamma r}\sum_{j\ge0}(1-a_j)
       |F_{\mu,\gamma}(a_j)|^r
 \lesssim\|F_{\mu,\gamma}\|_{H^r}^r.
\]
The sum includes $j=0$, so the total mass $m_0=\mu([0,1))$
is also controlled. This proves the assertion.
\end{proof}

\begin{lemma}\label{lem:moment-tail}
For $0<r<\infty$, $\gamma>0$, and a finite positive measure $\mu$ on
$[0,1)$, put $\sigma=\gamma-1/r$ and use $a_j,m_j,d_j$ as above. Then
\begin{equation}\label{eq:moment-tail-equivalence}
 \sum_{n=0}^\infty(n+1)^{r\gamma-2}\mu_n^r
 \asymp
 \sum_{j=0}^\infty d_j^r.
\end{equation}
\end{lemma}

\begin{proof}
\emph{Case $\sigma<0$.} In this case, $r\gamma<1$. Since $m_j \leq =m_0=\mu_0$ and $\mu_n\le\mu_0$ for all $n\geq 0$. This implies that 
\[
 m_0^r\le\sum_j2^{jr\sigma}m_j^r
 \le\frac{m_0^r}{1-2^{r\sigma}},
\]
and 
\[
 \mu_0^r\le\sum_n(n+1)^{r\gamma-2}\mu_n^r
 \le\mu_0^r\sum_n(n+1)^{r\gamma-2}<\infty.
\]
Thus both expressions are comparable to $\mu_0^r$, proving the
claim in this case.

\emph{Case $\sigma\geq 0$.}
For $j\ge0$ and $2^j\le n<2^{j+1}$,
$\mu_{2^{j+1}}\le\mu_n\le\mu_{2^j}$. Hence
\[
 2^{j(r\gamma-1)}\mu_{2^{j+1}}^r
 \lesssim
 \sum_{n=2^j}^{2^{j+1}-1}n^{r\gamma-2}\mu_n^r
 \lesssim
 2^{j(r\gamma-1)}\mu_{2^j}^r.
\]
This gives
\begin{equation}\label{eq:moment-block}
 \sum_{n=0}^\infty(n+1)^{r\gamma-2}\mu_n^r
 \asymp
 \mu_0^r+\sum_{j\ge1}2^{jr\sigma}\mu_{2^j}^r.
\end{equation}
For $j\ge1$,
\[
 \mu_{2^j}\ge(1-2^{-j})^{2^j}m_j\ge\tfrac14m_j,
\]
while $d_0=m_0=\mu_0$. Consequently,
\begin{equation}\label{eq:tail-to-moment}
 \sum_{j\ge0}d_j^r
 \lesssim \mu_0^r+\sum_{j\ge1}2^{jr\sigma}\mu_{2^j}^r.
\end{equation}

For the reverse estimate, we assume that $d\in\ell^r$. For $n\ge 1$, it is clear that
\[
 \mu_n=n\int_0^1t^{n-1} \mu([t,1))\,dt,\qquad n\ge1.
\]
Fix $j\ge1$. If $t\in[a_k,a_{k+1})$, then
\[
 1-t>2^{-k-1},\qquad \mu([t,1))\le m_k,\qquad
 a_{k+1}-a_k=2^{-k-1}.
\]
Since $t\le e^{-(1-t)}$ and $2^j-1\ge2^{j-1}$,
\[
 t^{2^j-1}
 \le\exp\bigl(-(2^j-1)(1-t)\bigr)
 \le\exp\left(-\tfrac14\,2^{j-k}\right).
\]
Thus
\begin{align*}
 \mu_{2^j}
 &=2^j\sum_{k\ge0}\int_{a_k}^{a_{k+1}}
       t^{2^j-1} \mu([t,1))\,dt\\
 &\le\frac12\sum_{k\ge0}
       2^{j-k}\exp\left(-\tfrac14\,2^{j-k}\right)m_k.
\end{align*}
Extend $d$ by zero to the negative integers. Multiplication by
$2^{j\sigma}$ yields
\[
 2^{j\sigma}\mu_{2^j}\le\tfrac12(K*d)_j,
 \qquad
 K_\ell=2^{\ell(1+\sigma)}e^{-2^\ell/4},\quad\ell\in\mathbb Z.
\]
 It is easy to  check that  $K\in\ell^u(\mathbb Z)$ for
every $u>0$. If $r\ge1$, apply Lemma~\ref{lem:discrete-young}\textup{(a)}.
If $r<1$, apply part \textup{(b)}. In both cases,
\[
 \sum_{j\ge1}2^{jr\sigma}\mu_{2^j}^r
 \le2^{-r}\|K*d\|_{\ell^r(\mathbb Z)}^r
\lesssim\|d\|_{\ell^r}^r.
\]
Together with $\mu_0=d_0$, \eqref{eq:moment-block}, and
\eqref{eq:tail-to-moment}, this proves
\eqref{eq:moment-tail-equivalence}. 
\end{proof}

\subsection{The operator implications}\label{subsec:lower-proof}\hfill\\\indent

\begin{proof}[Proof of Theorem~\ref{thm:lower-main}]
Fix $0<q<p\le\infty$ and $\gamma>0$, and use $r,\sigma,a_j,m_j,d_j$
as in the theorem. 
Lemma~\ref{lem:potential} gives \textup{(iii)}$\Leftrightarrow$
\textup{(iv)}$\Leftrightarrow$\textup{(v)}.
Lemma~\ref{lem:generating-tail} and Lemma~\ref{lem:moment-tail}
show that  \textup{(iii)}$\Leftrightarrow$ \textup{(vi)} $\Leftrightarrow$ \textup{(vii)} with the corresponding norm comparisons. Since  \textup{(ii)}$\Rightarrow$ \textup{(i)} is obvious, it remains to prove that \textup{(iii)}$\Rightarrow$ \textup{(ii)} and  \textup{(i)}$\Rightarrow$ \textup{(iii)}.

\smallskip
\noindent\emph{ \textup{(iii)} $\Rightarrow$ \textup{(ii)}.}
Assume $d\in\ell^r$. Then $P_{\mu,\gamma}\in L^r(\T)$ by
Lemma~\ref{lem:potential}. For $f\in H^p$, let
\[
 f^*(e^{i\theta})=\sup_{0\le u<1}|f(ue^{i\theta})|.
\]
The radial maximal estimate gives
$\|f^*\|_{L^p(\T)}\lesssim \|f\|_{H^p}$ when $0<p<\infty$;
for $p=\infty$ the same bound holds with constant $1$.
By \eqref{eq:radial-denominator}, for $0<\rho<1$,
\[
 |\mathcal C_{\mu,\gamma}f(\rho e^{i\theta})|
 \le2^\gamma f^*(e^{i\theta})P_{\mu,\gamma}(e^{i\theta})
 \quad\text{for almost every }\theta.
\]
Since $1/q=1/p+1/r$, by H\"older's inequality with exponents  and the maximal estimate we have
\begin{equation}\label{eq:lower-operator-upper}
 \|\mathcal C_{\mu,\gamma}\|_{H^p\to H^q}
 \lesssim\|P_{\mu,\gamma}\|_{L^r(\T)}
 \lesssim\|d\|_{\ell^r}.
\end{equation}
This includes $p=\infty$, where $r=q$. 

To obtain compactness, fix $0<R<1$ and write
\[
 \mu_R=\mu|_{[0,R]},\qquad
 \widetilde\mu_R=\mu|_{(R,1)}.
\]
The operator $\mathcal C_{\mu_R,\gamma}:H^p\to H^q$ is compact.
Indeed,  let $\{f_n\}$ be a bounded sequence in $H^p$ that converges to $0$ uniformly on
compact subsets of $\D$. Then 
\[
 \norm{\mathcal C_{\mu_R,\gamma}f_n}_{H^q}
 \le\norm{\mathcal C_{\mu_R,\gamma}f_n}_{H^\infty}
 \le\frac{\mu([0,R])}{(1-R)^\gamma}
 \sup_{|w|\le R}|f_n(w)|\longrightarrow0.
\]
Hence the truncated operator is compact.

For the  measure $\widetilde\mu_R$,
\[
 0\le P_{\widetilde\mu_R,\gamma}\le P_{\mu,\gamma}\in L^r(\T).
\]
For each nonzero $\theta\in[-\pi,\pi]$, the kernel in the potential
is uniformly bounded in $t$ by \eqref{eq:kernel-geometry}.
Since $\mu((R,1))\to0$, it follows that
$P_{\widetilde\mu_R,\gamma}(e^{i\theta})\to0$ as $R\to1^-$.
By the dominated convergence theorem together with $r<\infty$, an application of
\eqref{eq:lower-operator-upper} to $\widetilde\mu_R$ yields
\begin{equation*}\label{eq:lower-compact-approximation}
 \begin{aligned}
 \|\mathcal C_{\mu,\gamma}-\mathcal C_{\mu_R,\gamma}\|_{H^p\to H^q}
 =\|\mathcal C_{\widetilde\mu_R,\gamma}\|_{H^p\to H^q}
 \lesssim \|P_{\widetilde\mu_R,\gamma}\|_{L^r(\T)}\longrightarrow0.
 \end{aligned}
\end{equation*}
Thus $\mathcal C_{\mu,\gamma}$ is a norm limit of compact operators,
which proves \textup{(ii)}.

\smallskip
\noindent\emph{ \textup{(i)}$\Rightarrow$ \textup{(iii)}.}
 Assume $ N:=\|\mathcal C_{\mu,\gamma}\|_{H^p\to H^q}<\infty.$
If $p=\infty$, then $r=q$. Testing with the constant function $1$ gives
\[
 \|F_{\mu,\gamma}\|_{H^r}
 =\|\mathcal C_{\mu,\gamma}(1)\|_{H^q}\le N.
\]
Lemma~\ref{lem:generating-tail} therefore yields
$\|d\|_{\ell^r}\lesssim N$.

Suppose now that $p<\infty$. Let $c=\{c_j\}_{j\ge0}\in\ell^p$
be nonnegative and set
\[
 f=\sum_{j\ge0}c_jk_j,
\]
with $k_j$ as in \eqref{eq:kj}. Lemma~\ref{lem:synthesis} gives
convergence in $H^p$ and locally uniform absolute convergence, as well as
\begin{equation}\label{eq:f-synthesis}
 \|f\|_{H^p}\lesssim\|c\|_{\ell^p}.
\end{equation}
For fixed $j$,  $f(ta_j)\ge c_jk_j(ta_j)$.
If $t\in[a_j,1)$, then
\[
 1-a_j^2\ge1-a_j,\qquad
 1-ta_j^2\le3(1-a_j),\qquad
 1-ta_j\le2(1-a_j).
\]
Consequently,
\begin{equation}\label{eq:point-lower}
 \begin{aligned}
 \mathcal C_{\mu,\gamma}f(a_j)
 &\ge c_j\int_{[a_j,1)}
       \frac{k_j(ta_j)}{(1-ta_j)^\gamma}\,d\mu(t)\\
 &\gtrsim
       c_j\frac{m_j}{(1-a_j)^{\gamma+1/p}}.
 \end{aligned}
\end{equation}
As $d_j=m_j/(1-a_j)^{\gamma+1/p-1/q}$,
Lemma~\ref{lem:sampling}, \eqref{eq:point-lower} and
\eqref{eq:f-synthesis} give
\begin{align*}
 \|\{d_jc_j\}\|_{\ell^q}
 &\lesssim
   \left(\sum_{j\ge0}(1-a_j)
   |\mathcal C_{\mu,\gamma}f(a_j)|^q\right)^{1/q}\\
 &\lesssim \|\mathcal C_{\mu,\gamma}f\|_{H^q}
 \lesssim N\|c\|_{\ell^p}.
\end{align*}
For an arbitrary complex sequence $x\in\ell^p$, apply this inequality
to $c_j=|x_j|$; since $d_j\ge0$, it gives the same bound for
$\{d_jx_j\}$. Lemma~\ref{lem:diagonal} now implies $d\in\ell^r$.
In both cases we have proved
\begin{equation}\label{eq:lower-operator-lower}
 \|d\|_{\ell^r}
 \lesssim\|\mathcal C_{\mu,\gamma}\|_{H^p\to H^q}.
\end{equation}

Estimates \eqref{eq:lower-operator-upper} and
\eqref{eq:lower-operator-lower} show that
\[
 \|\mathcal C_{\mu,\gamma}\|_{H^p\to H^q}
 \asymp\|d\|_{\ell^r}.
\]
Combining this with Lemmas~\ref{lem:potential},
\ref{lem:generating-tail}, and~\ref{lem:moment-tail} gives exactly
\eqref{eq:norm-equivalence-lower}. The total mass term in the
continuous expression is retained through
\eqref{eq:integral-dyadic-equivalence}. This completes all seven
equivalences and the quantitative assertion.
\end{proof}

\begin{corollary}\label{cor:gamma-threshold}
Let $0<q<\infty$ and $\mu$ be a finite positive Borel measure on $[0,1)$.
\begin{enumerate}[label=\textup{(\roman*)}]
 \item If $0<\gamma<1/q$, then $\mathcal C_{\mu,\gamma}:H^\infty\to H^q$ is compact for every finite $\mu$.
 \item If $\gamma=1/q$, boundedness and compactness are equivalent to
 \[
 \int_0^1\frac{\mu([t,1))^q}{1-t}\,dt<\infty,
 \quad\text{equivalently}\quad
 \sum_{n\ge0}\frac{\mu_n^q}{n+1}<\infty.
\]
 \item If $\gamma>1/q$, boundedness is characterized by the tail
condition \eqref{eq:Hinfty-tail} and implies compactness.
\end{enumerate}
\end{corollary}

\begin{proof}
If $q\gamma<1$, then
\[
 \int_0^1\frac{\mu([t,1))^q}{(1-t)^{q\gamma}}\,dt
 \le\mu([0,1))^q\int_0^1(1-t)^{-q\gamma}\,dt<\infty.
\]
The remaining statements are immediate from Theorem~\ref{thm:lower-main}.
\end{proof}

\section{The target \texorpdfstring{$H^\infty$}{H-infinity}}\label{sec:target-infty}

The measure conditions in Theorem~\ref{thm:target-infty} follow
from the comparison lemmas in the preceding section.
With $p',s,b$ and $G_{\mu,\gamma}$ as in the theorem, observe that
\begin{equation*}\label{eq:target-exponent-shift}
 (\gamma+1)-\frac1{p'}=\gamma+\frac1p=s,
 \qquad p's+1=p'(\gamma+1).
\end{equation*}
Apply Lemmas~\ref{lem:potential}, \ref{lem:generating-tail}, and
\ref{lem:moment-tail} with the exponent and kernel parameter
$(r,\gamma)$ replaced by $(p',\gamma+1)$.
They give the equivalence of \textup{(iii)--(vi)} in
Theorem~\ref{thm:target-infty}, together with all comparisons
in \eqref{eq:target-norm} that do not involve the operator norm.
It remains only to prove
\[
 \textup{(v)}\Longrightarrow\textup{(ii)}
 \Longrightarrow\textup{(i)}\Longrightarrow\textup{(iii)}
\]
and the  operator norm estimates.

\begin{proof}[Proof of Theorem~\ref{thm:target-infty}]
\noindent\emph{ \textup{(v)}$\Longrightarrow$\textup{(ii)}.}
Assume $F_{\mu,\gamma+1}\in H^{p'}$.
For $h\in H^p$ and $G\in H^{p'}$, the Hadamard product satisfies
\begin{equation}\label{eq:hadamard-to-infty}
 \|h\star G\|_{H^\infty}\le\|h\|_{H^p}\|G\|_{H^{p'}}.
\end{equation}

By the normalization of $R^\gamma$, 
$F_{\mu,\gamma+1}=R^\gamma F_\mu$.
Thus the factorization \eqref{eq:C-factorization} and \eqref{eq:hadamard-to-infty} give
\begin{equation}\label{eq:target-upper}
 \|\mathcal C_{\mu,\gamma}\|_{H^p\to H^\infty}
 \le\|C^\gamma\|_{H^p\to H^p}
       \|F_{\mu,\gamma+1}\|_{H^{p'}}
 \lesssim\|F_{\mu,\gamma+1}\|_{H^{p'}}.
\end{equation}
Since $p'<\infty$, polynomials are dense in $H^{p'}$; see
\cite{Duren1970}. Choose polynomials $Q_N$ such that
$\|F_{\mu,\gamma+1}-Q_N\|_{H^{p'}}\to0$, and define
\[
 T_Nf=Q_N\star C^\gamma f.
\]
Each $T_N$ has finite-dimensional range, since its range is contained in the space of polynomials of degree at most $\deg Q_N$.
It follows from the  estimate \eqref{eq:hadamard-to-infty}  that
\begin{equation*}\label{eq:target-finite-rank}
 \|\mathcal C_{\mu,\gamma}-T_N\|_{H^p\to H^\infty}
 \le\|C^\gamma\|_{H^p\to H^p}
       \|F_{\mu,\gamma+1}-Q_N\|_{H^{p'}}\longrightarrow0.
\end{equation*}
Hence $\mathcal C_{\mu,\gamma}$ is compact. This proves
\textup{(v)}$\Rightarrow$\textup{(ii)}.
It remains to prove \textup{(i)} $ \Longrightarrow$\textup{(iii)}.

\smallskip
\noindent\emph{\textup{(i)} $ \Longrightarrow$\textup{(iii)}.}
Suppose
\[
 \mathcal N:=\|\mathcal C_{\mu,\gamma}\|_{H^p\to H^\infty}<\infty.
\]
Let $c=\{c_j\}_{j\ge0}$ be a finitely supported nonnegative
sequence, and set $f=\sum_jc_jk_j$, with $k_j$ from
\eqref{eq:kj}. Lemma~\ref{lem:synthesis} gives
$\|f\|_{H^p}\lesssim\|c\|_{\ell^p}$.
For $0<x<1$, the integrand defining
$\mathcal C_{\mu,\gamma}f(x)$ is nonnegative. Letting $x\to1^-$
and applying Fatou's lemma yields
\begin{equation}\label{eq:target-boundary-functional}
 \int_{[0,1)}\frac{f(t)}{(1-t)^\gamma}\,d\mu(t)
 \le\liminf_{x\to1^-}\mathcal C_{\mu,\gamma}f(x)
 \le\mathcal N\|f\|_{H^p}
 \lesssim_p\mathcal N\|c\|_{\ell^p}.
\end{equation}
On the other hand, 
\begin{align*}
 \int_{[0,1)}\frac{f(t)}{(1-t)^\gamma}\,d\mu(t)
 &\ge\sum_jc_j\int_{[a_j,1)}
       \frac{k_j(t)}{(1-t)^\gamma}\,d\mu(t)\\
 &\ge2^{-1/p}\sum_jc_j
       \frac{\mu([a_j,1))}{(1-a_j)^{\gamma+1/p}}
 =2^{-1/p}\sum_jb_jc_j.
\end{align*}
Together with \eqref{eq:target-boundary-functional}, this gives
\[
 \sum_jb_jc_j\lesssim_p\mathcal N\|c\|_{\ell^p}
\]
for every finitely supported nonnegative sequence $c$.
Replacing $c_j$ by $|x_j|$ and passing to finite truncations
shows that $D_b:x\mapsto\{b_jx_j\}$ is bounded from $\ell^p$
to $\ell^1$.  By Lemma~\ref{lem:diagonal}   we obtain
\begin{equation}\label{eq:target-lower}
 b\in\ell^{p'},\qquad
 \|b\|_{\ell^{p'}}\lesssim
 \|\mathcal C_{\mu,\gamma}\|_{H^p\to H^\infty}.
\end{equation}
This proves \textup{(i)}$\Rightarrow$\textup{(iii)}.
Finally, combine \eqref{eq:target-upper} and
\eqref{eq:target-lower} with the measure comparisons preceding
the proof to obtain \eqref{eq:target-norm}.
The total mass term is inherited from
\eqref{eq:integral-dyadic-equivalence}, so an atom at $0$ is
included. All six assertions and the norm comparisons follow.
\end{proof}

\begin{corollary}\label{cor:target-shift}
Let $1<p<\infty$, $\gamma>0$, and let $\mu$ be a finite positive
Borel measure on $[0,1)$. Then
\[
 \mathcal C_{\mu,\gamma}:H^p\to H^\infty\ \text{is bounded}
 \quad\Longleftrightarrow\quad
 \mathcal C_{\mu,\gamma+1}:H^p\to H^1\ \text{is bounded}.
\]
In either case both operators are compact, and
\begin{equation*}\label{eq:target-shift-norm}
 \|\mathcal C_{\mu,\gamma}\|_{H^p\to H^\infty}
 \asymp
 \|\mathcal C_{\mu,\gamma+1}\|_{H^p\to H^1}.
\end{equation*}
\end{corollary}

\begin{proof}
    Take  $q=1$ and
kernel parameter $\gamma+1$ in Theorem~\ref{thm:lower-main},  then 
the criterion and its norm are those of
Theorem~\ref{thm:target-infty}.
\end{proof}

\begin{example}\label{ex:target-carleson-fails}
Suppose  $1<p<\infty$ and $\gamma>0$, and put $s=\gamma+1/p$.
For $d\mu(t)=(1-t)^{s-1}\,dt$, it is clear that 
\[
 \mu([t,1))=\frac{(1-t)^s}{s},\qquad
 b_j=2^{js}\mu([a_j,1))=\frac1s\quad(j\ge0).
\]
Thus $\mu$ is an $s$-Carleson measure, but $b\notin\ell^{p'}$.
Theorem~\ref{thm:target-infty} implies that
$\mathcal C_{\mu,\gamma}:H^p\to H^\infty$ is not bounded.
In particular, the  Carleson criterion in Theorem \ref{thm:upper-improved} cannot be
extended to $q=\infty$ when $1<p<\infty$ without the additional
summability requirement.
\end{example}

\section{Sharpness and threshold examples}\label{sec:classical}

\begin{proposition}\label{prop:strong-condition}
Let $1\le q<p<\infty$, $r=pq/(p-q)$, and $s=1-1/r$. Then
\begin{equation}\label{eq:integral-sum-equivalence}
 \int_0^1\frac{d\mu(t)}{(1-t)^s}<\infty
 \quad\Longleftrightarrow\quad
 \sum_{n=0}^\infty(n+1)^{-1/r}\mu_n<\infty.
\end{equation}
Either condition implies that $\mathcal C_\mu:H^p\to H^q$ bounded, but neither is necessary.
\end{proposition}

\begin{proof}
Since
\[
 (1-t)^{-s}=\sum_{n=0}^\infty
 \frac{\Gamma(n+s)}{\Gamma(s)\Gamma(n+1)}t^n
\]
with coefficients comparable to $(n+1)^{s-1}=(n+1)^{-1/r}$,
Fubini's theorem gives \eqref{eq:integral-sum-equivalence}.
The implication to boundedness was obtained in
\cite[Theorem~3.15(ii)]{Blasco2024}. It also follows directly from
Theorem \ref{thm:lower-main} as follows. Suppose
\[
 I_s:=\int_{[0,1)}(1-t)^{-s}\,d\mu(t)<\infty,
 \qquad s=1-\frac1r>0,
\]
and put $d_j=2^{js}\mu([a_j,1))$, where $a_j=1-2^{-j}$.
By Fubini's theorem  we have
\begin{align*}
 \sum_{j=0}^\infty d_j
 &=\int_{[0,1)}
   \left(\sum_{j:\,a_j\le t}2^{js}\right)d\mu(t)\\
 &\le C_s\int_{[0,1)}(1-t)^{-s}\,d\mu(t)=C_s I_s.
\end{align*}
Consequently, $d\in\ell^1\subset\ell^r$, and
Corollary~\ref{thm:complete-classical} implies boundedness (and compactness).
Proposition~\ref{prop:strictness} below proves that the condition is
not necessary.
\end{proof}

\begin{proposition}\label{prop:strictness}
Let $1\le q<p<\infty$, $r=pq/(p-q)$ and $s=1-1/r$. Fix $1/r<\beta\le1$ and define
\begin{equation*}\label{eq:log-measure}
 d\mu_\beta(t)
 =
 (1-t)^{s-1}
 \left(\log\frac{e}{1-t}\right)^{-\beta}\,dt.
\end{equation*}
Then $\mathcal C_{\mu_\beta}:H^p\to H^q$ is bounded and compact, while
\[
 \int_0^1\frac{d\mu_\beta(t)}{(1-t)^s}=\infty.
\]
Thus \eqref{eq:strong-integral} and \eqref{eq:blasco-sufficient} are strictly stronger than the exact lower-triangle condition.
\end{proposition}

\begin{proof}
A direct integration gives, as $t\to1^-$,
\begin{equation*}\label{eq:tail-log}
 \mu_\beta([t,1))
 \asymp
 (1-t)^s\left(\log\frac{e}{1-t}\right)^{-\beta}.
\end{equation*}
 Hence the exact tail condition on $(0,1)$ is
\[
\int_{0}^1
 \left(\log\frac{e}{1-t}\right)^{-\beta r}
 \frac{dt}{1-t}<\infty,
\]
which holds exactly when $\beta r>1$.  On the other hand,
\[
 \int_{0}^1\frac{d\mu_\beta(t)}{(1-t)^s}
 =\int_{0}^1
 \frac1{1-t}
 \left(\log\frac{e}{1-t}\right)^{-\beta}\,dt,
\]
which diverges for $\beta\le1$.
\end{proof}

\begin{example}\label{ex:quasi-critical}
Let $0<q<p\le\infty$, put $1/r=1/q-1/p$, and take
$\gamma=1/r$. For $\beta>0$, define the finite positive measure
\[
 d\mu_\beta(t)
 =\frac{\beta}{(1-t)[\log(e/(1-t))]^{\beta+1}}\,dt.
\]
The substitution $x=\log(e/(1-t))$ gives exactly
\[
 \mu_\beta([t,1))=[\log(e/(1-t))]^{-\beta},
 \qquad \mu_\beta([0,1))=1.
\]
Consequently,
\[
 \int_0^1\frac{\mu_\beta([t,1))^r}{1-t}\,dt
 =\int_1^\infty x^{-\beta r}\,dx.
\]
Theorem~\ref{thm:lower-main} implies that
$\mathcal C_{\mu_\beta,\gamma}:H^p\to H^q$ is bounded, equivalently
compact, precisely when $\beta>1/r$. This conclusion remains valid
when $q<1$ and when $r<1$. Finiteness of $\mu$ alone does not
suffice on the critical line $\gamma=1/r$.
\end{example}

\section{Concluding remarks}\label{sec:conclusion}

Theorems~\ref{thm:upper-improved}, \ref{thm:lower-main}, and
\ref{thm:target-infty} separate three mapping regimes for the
generalized Ces\`aro operators. For finite target exponents in the
upper triangle, and for $0<p\le1$ with target $H^\infty$,
boundedness is governed by a pointwise Carleson condition.
For $0< q<p\le\infty$, it is governed instead by an
$\ell^r$ condition on the normalized tails, where $1/r=1/q-1/p$.
The generating-function and moment conditions are equivalent
forms of the same criterion, and $p=\infty$ is included by
setting $r=q$.

For the target endpoint $H^p\to H^\infty$, $1<p<\infty$,
the summability exponent is $p'$ and the relevant generating
function is $F_{\mu,\gamma+1}$. Thus this endpoint is not obtained
by simply substituting $q=\infty$ in the Carleson criterion of
Theorem~\ref{thm:upper-improved}. Boundedness and compactness
coincide both in the lower triangle and for this target endpoint.
For $\gamma=1$, Corollary~\ref{thm:complete-classical} gives a complete
boundedness classification for all  $0< p,q\le\infty$. 

 Further questions concern
characterizations for complex measures, where positivity is no
longer available, and analogous results on weighted Bergman and
mixed norm spaces.

\section*{Conflicts of Interest}
The authors declare that there is no conflict of interest.




\section*{Availability of data and materials}
Data sharing is not applicable to this article as no datasets were generated or analysed during
the current study: the article describes entirely theoretical research.

\section*{Funding}
The first author was supported by the  National Natural Science Foundation of China (No.1260014198). The second author was supported by  the  National Natural Science Foundation of China (No.12601226).

\end{document}